\documentclass[a4paper,12.5pt]{amsart}
\usepackage[margin=1.5in]{geometry}
\title{A converse theorem for degree 2 elements of the Selberg class with restricted gamma factor.}
\author{Michael Farmer}

\usepackage{bbm}

\usepackage{dirtytalk}
\usepackage{graphicx}
\usepackage{framed}
\usepackage{amsmath}
\usepackage{amssymb}
\usepackage{amsfonts}
\usepackage{cite}
\usepackage{mathrsfs}
\usepackage{array}
\usepackage{fullpage}
\usepackage [english]{babel}
\usepackage [autostyle, english = american]{csquotes}
\MakeOuterQuote{"}

\usepackage{cite}
\usepackage{amsthm} 
\usepackage{url}
\newtheorem{thm}{Theorem}[section]
\newtheorem{lem}[thm]{Lemma}

\theoremstyle{definition}

\theoremstyle{remark}
\newtheorem*{rem}{Remark}

\newcommand{\ind}{\mathbbm{1}_{ \chi= \chi_0}}

\newcommand{\F}{L_f}

\newcommand{\Mod}[1]{\ (\mathrm{mod}\ #1)}

\begin{document}
\maketitle
\begin{abstract}
    We  prove a converse theorem for a family of L functions of degree 2  with gamma factor coming from a holomorphic cuspform. We show these L functions coincide with either those coming from a newform or a product of L functions arising from Dirichlet characters. We require some analytic data on the Euler factors, but don't require anything on the shape. We also suppose that the twisted L functions satisfy expected functional equations. We incorporate the ideas from \cite{booker2014weil} so that the non-trivial twists are allowed to have arbitrary poles. 
\end{abstract}
\section{Introduction}
In \cite{Converse}, Booker proves a version of Weil's converse theorem without needing knowledge of the root number. We extend on his work, by not assuming a specific shape of the Euler factors. However we do still need to have some analytic data on the Euler factors, namely holomorphy and non-vanishing on and to the right of the line $\Re(s)=1/2$. This is a natural restriction to make when looking from the perspective of the Selberg class (see  \cite{perelli2005survey} for details on the Selberg class).
Our result also makes partial progress of extending the result of Kaczorowski, Perelli in \cite{kaczorowski2018hecke} to general conductor. In their paper, they get their result for conductor $5$ without any information on twists through the theory of non-linear twists of L functions. Combining our theorem with a result classifying the gamma factors of degree 2 elements of the Selberg class (as in \cite{kaczorowski2020classification} for conductor $1$) , we could prove a more general theorem.   \\
Let $S_k ^{\text{new}} (\Gamma_0 (N) , \xi)$ denote the newforms of weight $k$, level $N$ and  nebentypus character $\xi$. 
We prove the following theorem. 
\begin{thm}
\label{Main thm}
Let $\{a_n \}_{n=1}^{\infty}$ be a multiplicative sequence of complex numbers satisfying $a_n= O(n^{\lambda})$ for some $\lambda \in \mathbb{R}_{>0}$ and such that $\sum_{j=0}^{\infty} a_{p^j} p^{-js} $  is analytic and non-vanishing for $\Re(s) \geq 1/2$ and every prime $p$. Fix positive  integers $k,N$. For any  primitive Dirichlet character $\chi$ of conductor $q$ coprime to $N$, define 
\[  \Lambda_{\chi} (s)= \Gamma_{\mathbb{C}} \left(  s+ \frac{k-1}{2}  \right) \sum_{n=1}^{\infty} a_n  \chi(n) n^{-s}                      \]
for $\Re(s)> 1 + \lambda$, where $\Gamma_{\mathbb{C}} (s) =  2 (2 \pi)^{-s} \Gamma(s)$. Suppose, for every such $\chi$, that $\Lambda_{\chi}(s)$
continues to a meromorphic function on $\mathbb{C}$ and satisfies the functional equation
\begin{equation}{\label{Functional equation for twists}}
     \Lambda_{\chi}(s) = \varepsilon_{\chi} (Nq^2)^{1/2 -s} \overline{ \Lambda_{\chi}  ( 1 - \overline{s} )}          , 
\end{equation}  
for some $\varepsilon_{\chi} \in \mathbb{C}$. Let $\mathbf{1}$ denote the character of modulus $1$ and suppose there is a nonzero polynomial $P$ such that $P(s) \Lambda_{\mathbf{1}} (s)$ continues to an entire function of finite order. 

Then one of the following holds: \\
(i) $k=1$ and there are primitive characters $\xi_1 \pmod{N_1}$ and $\xi_2 \pmod{N_2}$ such that $N_1 N_2 = N$ and $  \sum_{n=1}^{\infty} a_n n^{-s} =  L(s, \xi_1) L(s , \xi_2)$, where $ L(s, \xi_1), L(s, \xi_2)$ are the usual Dirichlet L functions. \\
(ii) $  \sum_{n=1}^{\infty} a_n n^{\frac{k-1}{2}} e(nz)  \in S_k ^{\text{new}} (\Gamma_0 (N) , \xi)$ for some Dirichlet character $\xi$ of conductor  dividing $N$. 
\end{thm}
\begin{rem}
     If we suppose that for every $\chi$ of conductor $q \equiv 1 \pmod{N}$ that $\Lambda_{\chi} (s)$ is entire, then we don't require \eqref{Functional equation for twists} for  characters of conductor not congruent to $1 \pmod{N}$. If $\Lambda_{\chi}$ is not entire then we use this extra information in Lemma \ref{Twists are entire}. 
\end{rem}
\begin{rem}
A similar result should be true for L functions with gamma factors coming from Maass forms, but would be more difficult. It requires analysis of hypergeometric functions as in \cite{neururer2018weil}.
\end{rem}
From now on we shall assume the hypothesis of theorem \ref{Main thm}.
\section{Analysis of Euler factors.}
We shall assume for now that $\Lambda_{\mathbf{1}} (s)$ is entire and deal with the meromorphic case at the end. \\
By the conditions in theorem \ref{Main thm} have the following lemma to constrain the poles of $\Lambda_{\chi}$, shown in the proof of  \cite[ Theorem 1.1]{booker2014weil}.  
\begin{lem}{\label{Twists are entire}}
The function $\Lambda_{\chi} (s)$ is entire of finite order for every primitive character $\chi$ of prime conductor $q \nmid N$. 
\end{lem}
For $ \chi$ a primitive character of conductor $q \nmid N$, both $\Lambda_1 (s)$ and $\Lambda_{\chi}(s)$  are entire of finite order. The Phragm\'en-Lindel\"of convexity principle means they are bounded in vertical strips. \\
Let $\mathbb{H}= \{  z \in \mathbb{C}: \Im(z)>0     \}$ denote the upper half plane. Let $e(s):= e^{ 2 \pi i s}$ for $s \in \mathbb{C}$.  For $z \in \mathbb{H}$, set 
\[  f_n= a_n n^{\frac{k-1}{2}},  \; \; \; \; \;     f(z)=  \sum_{n=1}^{\infty} f_n  e(nz), \; \; \; \; \;                  \overline{f}(z) = \sum_{n=1}^{\infty} \overline{f_n}  e(nz)    .       \]
Recall the $k$-slash operator defined for any function $g: \mathbb{H} \rightarrow \mathbb{C}$ and any matrix $ \gamma= \begin{pmatrix}
a & b \\
c & d \\
\end{pmatrix}  $  of positive determinant by 
\[  (g| \gamma) (z) = (\det \gamma)^{k/2} (cz+d)^{-k} g\left(   \frac{az+b}{cz+d}  \right)      .    \]  For matrices $\gamma, \gamma'$, we write  $\gamma \simeq_{g} \gamma'$, if $ g|\gamma= g|\gamma'$. \\ 
The functional equation  \eqref{Functional equation for twists} for $\chi= \mathbf{1}$ and Hecke's argument (see Theorem 4.3.5 from \cite{miyake2006modular})  
implies that $ f| H_N = \varepsilon_1 i^k \overline{f}$ , where $ H_N= \begin{pmatrix}
0 & -1 \\
N & 0 \\
\end{pmatrix}   $. Let $ P=\begin{pmatrix}
1 & 1 \\
0 & 1 \\  \end{pmatrix} $ and $I_2$ be the identity matrix. Then since both $f$ and $\overline{f}$ are Fourier series, 
$P \simeq_{f} I_2$ and $P \simeq_{\overline{f}} I_2$.  \\

Consider the additive twist by $\alpha$, defined by $F(s,\alpha)= \sum_{n=1}^{\infty} a_n n^{-s} e(n\alpha)$ and $\overline{F} (s,\alpha)=  \sum_{n=1}^{\infty}  \overline{a_n} n^{-s} e(n\alpha)  $. Moreover denote $\F (s):= F(s,0)$ and $ L_{\overline{f}}( s):= \overline{ F} (s,0)$. Let $q  \equiv 1 \pmod{N} $ be a prime number. 
For $\chi$ a Dirichlet character mod $q$, the Gauss sum is defined by $ \tau(\chi)= \sum_{a=1}^{q} \chi(a) e( \frac{a}{q})$.
By Fourier analysis on $\mathbb{Z}/ q \mathbb{Z}$, 

 \[  e( \frac{n}{q}) =1 - \frac{q}{q-1}  \chi_0 (n) + \frac{1}{q-1} \sum_{ \substack{  \chi \Mod q  \\ 
 \chi \neq \chi_0   }     }    \tau (\overline{ \chi}) \chi(n) .  \]

 Multiplying  by $a_n n^{-s}$ and summing, we get the following relationship between additive twists and multiplicative twists.  
 \begin{equation}{\label{Twist in terms of non twists}}
      F(s, \frac{1}{q}) = \F(s) - \frac{q}{q-1} \frac{\F(s)}{F_q(s) } + \frac{1}{q-1} \sum_{ \substack{  \chi \Mod q  \\ 
 \chi \neq \chi_0   }     }  \tau(\overline{\chi}) F(s, \chi)   ,    
 \end{equation}  
 where $F_q(s) =  \sum_{j=0}^{\infty} a_{q^j} q^{-js} $ and $F(s, \chi)= \sum_{n=1}^{\infty} a_n \chi(n) n^{-s}$.  
 \\
We shall prove the following lemma.
\begin{lem}
\label{Euler Factors}
Let $q \equiv 1 \pmod{N}$ be a sufficiently large  prime number. Then the function $F_q(s)^{-1}$ is a polynomial in $q^{-s}$ of degree $\leq 2$. 
\end{lem}
Let $\Im(s)=t$.  Because $F_q(s)^{-1}$  is $ \frac{ 2 \pi i}{\log q}$-periodic, for the rest of this section we shall assume $t>0$ is contained in the interval $\left[ \frac{ 2 \pi }{\log q}, \frac{ 4 \pi }{\log q} \right] $.
\begin{proof}

By Mellin transformations, 
\begin{equation}
{\label{Mellin transformation}} \frac{      F(s, \frac{1}{q})}{  \F(s)}= \frac{\int_{0}^{\infty} y^{s +\frac{k-1}{2} -1} f(\frac{1}{q}+iy)  \; dy}{ (2 \pi)^{-s - \frac{k-1}{2}} \Gamma(s+ \frac{k-1}{2}) \F(s) }  .    \end{equation}
This equation will allow us to meromorphically continue $F(s, \frac{1}{q})$. 
\\

We want to estimate $ \int_{0}^{\infty} y^{s +\frac{k-1}{2}-1}  f(\frac{1}{q}+iy)   \; dy$ as $\Re(s) \rightarrow - \infty$. For $y>1$, we can use the trivial estimate 
\begin{align}
    \left| f(\frac{1}{q}+iy) \right|  \leq \sum_{n=1}^{\infty} |a_n| n^{\frac{k-1}{2}} e^{- 2 \pi n y}
    \ll \sum_{n=1}^{\infty} n^{{\frac{k-1}{2}}+ \lambda} e^{-2 \pi n y} \ll e^{- 2 \pi y}.
\end{align}
 When $y \leq 1$ we will use the modularity relation $ f| H_N = \varepsilon_ {\bold{1}} i^k \overline{f}$ and the fact $q \equiv 1 \pmod{N}$. Write $q=MN+1$ ,  where $M$ is an integer. Then 

 $ \overline{ f}| P^M H_N =  \frac{ \overline{\varepsilon_{\bold{1}}}}{  i^k}  f$ since $ H_N^2 \simeq_{f} I_2$. We observe  that, 

\[H_N P^{M} H_N =   \begin{pmatrix}
-N & 0 \\
MN^2 & -N 
\end{pmatrix}  \simeq_{f}  \begin{pmatrix}
1 & 0 \\
-MN & 1 
\end{pmatrix} .   \]
Also this means, 
\begin{equation*}{\label{modular with respect to gamma_{q,1}}}
    P^{-1} H_N P^M H_N      \simeq_{f}  \begin{pmatrix}
 q & -1 \\
1-q & 1 
\end{pmatrix} .      
\end{equation*}     
Hence we  get a relationship under the action of 
\[  H_N    P^{-1} H_N P^M H_N    \simeq_{\overline{f}}   \begin{pmatrix}
q-1 & -1 \\
Nq & -N 
\end{pmatrix}=:  \gamma  ,      \]
namely 
\begin{align*}
    \overline{f}| \gamma  =  \omega f
\end{align*}
for some $ \omega$ with $|\omega|=1$. This can be written explicitly in the form

\begin{equation*}{\label{modular relation for f}}
    f(z)=   \overline{\omega} N^{ - \frac{k}{2}} (qz-1)^{-k} \overline{f} \left( \frac{(q-1)z -1}{ Nqz-N}     \right)         .  
\end{equation*}   
 Using this identity, we have
\[ f(\frac{1}{q}+iy)  \ll_{q}  \left(\frac{1}{y}\right)^{k} \sum_{n=1}^{\infty} |a_n| n^{\frac{k-1}{2}} e^{         \frac{ - 2 \pi n} {  q^2 N y} \left( 1+ O(y) \right)  }   \]
for $y \leq 1$  because 
\[      \frac{(q-1) ( \frac{1}{q} +iy)  -1}{ Nq ( \frac{1}{q} +iy  )-N} = \frac{-1}{ q^2 N (iy)} \left( 1 + O(y) \right)              \] 
for $y \leq 1$. Note that 
\[       \sum_{n=1}^{\infty} |a_n| n^{ \frac{k-1}{2}} e^{     \frac{ - 2 \pi n} {  q^2 N y} \left( 1 +O(y) \right)   } \ll     e^{ \frac{ -  2 \pi }{ q^2 N y}   }    .        \]
 Let $s= \sigma+it$ where $\sigma< -\frac{k-1}{2}$. Then 
\begin{align}{\label{integral estimate for f}}
    \int_{0}^{\infty} y^{s+\frac{k-1}{2}-1-k}  f(\frac{1}{q}+iy)  \; dy & \ll_{q}   \int_{0}^{1} y^{\sigma + \frac{k-1}{2} -1  }   e^{ \frac{ - 2 \pi  }{q^2 N y}   } \; dy + \int_{1}^{\infty}   y^{\sigma + \frac{k-1}{2} -1 } e^{- 2  \pi y} \; dy  \nonumber \\
        & \ll_q   \left( \frac{q^2 N}{ 2 \pi}  \right)^{| \sigma|}  \int_{\frac{2 \pi}{q^2 N}}^{\infty} y^{ -\sigma + \frac{k+1}{2} +1} e^{-y} \; dy                      +   \frac{ 1  }{ \left| \sigma + \frac{k-1}{2} \right|}   \nonumber \\ 
                                     \\& \ll_q  \left( \frac{q^2 N}{ 2  \pi}  \right)^{| \sigma|} \Gamma\left( \left| - \sigma + \frac{k+1}{2} \right| \right)     +    \frac{ 1  }{ \left| \sigma + \frac{k-1}{2} \right|}  \nonumber . \end{align} 
 The second term converges to zero as $\sigma \rightarrow -\infty$.
 By Stirlings formula 
 \[ \left| \Gamma(- \sigma + \frac{k+1}{2}) \right| \ll 
 \left( \frac{|\sigma|}{e} \right)^{|\sigma| +\frac{k}{2}}    .\] 
     By the functional equation \eqref{Functional equation for twists} for $\chi= \mathbf{1}$,  as $\sigma \rightarrow -\infty$
     \begin{align}{\label{numerator of additive twist}}  \left| \F(s) \Gamma(s+\frac{k-1}{2}) (2 \pi)^{-s} \right| &= \left| \varepsilon_{\bold{1}} \frac{\sqrt{N}}{ 2 \pi} \left( \frac{N}{2\pi}  \right)^{-s}  \Gamma(1-s +\frac{k-1}{2}) L_{\overline{f}}(1 - s)  \right| \nonumber \\
     & \gg _q \left(\frac{N}{ 2 \pi} \right)^{|\sigma|}    
      \left(\frac{|\sigma|}{e}    \right)^{|\sigma| +\frac{k}{2}},
     \end{align}
    where we use the fact $t$ is contained in the interval $\left[ \frac{ 2 \pi }{\log q}, \frac{ 4 \pi }{\log q} \right] $   . \\
   From \eqref{Mellin transformation}, \eqref{integral estimate for f} and \eqref{numerator of additive twist}, as $\sigma \rightarrow - \infty$
    \[      \frac{F(s, \frac{1}{q})}{\F(s)} \ll_q q^{2 |\sigma|} .  \]
    \\ 
    Moreover, by  the functional equation \eqref{Functional equation for twists}, 
    \[  \frac{F(s,\chi)}{\F(s)} = \frac{\Gamma_{\mathbb{C}} \left(s+ \frac{k-1}{2} \right) F(s,\chi) }{    \Gamma_{\mathbb{C}} \left(s+ \frac{k-1}{2} \right) \F(s)  }    \ll q^{2 |\sigma|} ,         \]
    as $\sigma \rightarrow -\infty$. 

 Let $E_q(q^{-s})= \frac{1}{F_q(s)}$. Since $F_q(s)= \sum_{j=0}^{\infty} a_{q^j} q^{-js}$, we have a power series expansion for $E_q(z)$. By the hypotheses on $F_q(s)$ the radius of convergence of $E_q(z)$ is at least $q^{-1/2}$. 
 To prove  lemma \ref{Euler Factors}, we need to show that $E_q(z)$ is a polynomial of degree $\leq 2$.  We know
 \[  F(s, \frac{1}{q})= \frac{  (2 \pi)^ {  \frac{k-1}{2}} \int_{0}^{\infty} y^{s + \frac{k-1}{2} -1}  f( \frac{1}{q}+iy) \; dy }{(2 \pi)^{-s} \Gamma(s+\frac{k-1}{2})   }    \]
 is entire  thanks to the first line of  \eqref{integral estimate for f} and since $\Gamma(s)$ has no zeroes. By \eqref{Twist in terms of non twists}, and our estimates above, $E_q(q^{-s})$ has a meromorphic continuation to $\mathbb{C}$, and satisfies

 \begin{equation}{\label{Estimate for small}}
     E_q(q^{-s}) \ll_q q^{2 |\sigma| } 
 \end{equation}     
  uniformly in $t$ as $\sigma \rightarrow -\infty$. Thus 
 
  \[ E_q(z) \ll_q |z|^{2}     \]
 for $|z|$ sufficiently large. Once we have shown $E_q(z)$ is entire we have finished the proof of the lemma.   \\ \\
 
 We shall show $E_q(q^{-s})$ is entire if $q$ is sufficiently large. \\ 
  Thanks to our assumption on $F_q(s)$ in Theorem \ref{Main thm}, $E_q(q^{-s})$ is analytic for $\Re(s) \geq 1/2$. Also, the estimate \eqref{Estimate for small} means it is also analytic in some left half plane. By $t$-periodicity, for large $q$ it suffices to show $E_q(q^{-s})$ is analytic in some region of the form 
  \[ D_K:= \{ s \in \mathbb{C}: \Re(s) \in (K, \frac{1}{2}), \Im(s) \in [0,1]                  \}                   ,    \]
  where $K< \frac{1}{2}$. By Lemma \ref{Twists are entire}, $F(s,\chi)$ is entire so the equation 
 \[ E_q(q^{-s})= \frac{q-1}{q}\left( 1 - \frac{F(s, \frac{1}{q})}{\F(s)} + \frac{1}{q-1} \sum_{ \substack{  \chi \Mod q  \\ 
 \chi \neq \chi_0   }     }  \tau(\overline{\chi})  \frac{F(s, \chi)}{\F(s)}                \right)         \]
implies the only possible  poles of $E_q(q^{-s})$ must come from zeroes of $\F$. There are only finitely many zeroes of $\F$ in the region $D_K$. By $t$ periodicity, if $E_q(q^{-s})$ has a pole at $s$ in  $D_K$, then $\F$ must also have zeroes at $s+ i \frac{ 2 \pi k }{\log q} $, for $k \in \mathbb{Z}$. For large enough $q$ we would have too many zeroes of $\F$ in the region $D_K$. Hence for large $q$, $E_q(q^{-s})$ cannot have any poles in $D_K$ as required.

 \end{proof}
\section{Analysis of the coefficients in the Euler factors}
We analyse the Euler factors using arguments from \cite{Converse}. \\
Let $q \equiv 1 \pmod{N}$ be a sufficiently large prime number ( as in lemma \ref{Euler Factors}), then $F_q(s)^{-1}$ can be expressed in the form  $ 1 - \lambda q^{-s} + \mu q^{-2s}$ where $\lambda= a_q$, $\mu=  a_q^2-a_{q^2}$.
Define
\[      \epsilon = \begin{cases}  \lambda/ \overline{\lambda} \; \; \; \text{if} \; \; \; \lambda \neq 0 ,\\
\frac{\mu}{|\mu|} \;  \; \; \text{if}   \; \; \; \lambda=0  , \mu \neq 0  ,\\
1  \;  \; \; \text{if}   \; \; \;   \lambda =0, \mu =0  .
\end{cases}
\] 
\begin{rem}We shall show that the third case can't happen so our choice here isn't important.
\end{rem} 

Let $r=1- \epsilon \overline{\mu}$, then $D_q(s)= r+q -1 -q( 1-\lambda q^{-s}+ \mu q^{-2s})$ satisfies a functional equation of the form $D_q(s)= \epsilon q^{1-2s} \overline{D_q(1-\overline{s})}$. Define $\Lambda_{c_q+r} (s)= \Lambda_{c_q} (s)+ r \Lambda_1(s)$, where $ \Lambda_{c_q} (s)= \Gamma_{\mathbb{C}} (s+ \frac{k-1}{2} )  \sum_{n=1}^{\infty} a_n c_q(n) n^{-s}$ and $c_q(n) =\sum_{ \substack{ a \pmod q  \\ (a,q)=1       } } e(\frac{an}{q})   $ is a Ramanujan sum.  \\ \\
Since $c_q(n)=q-1-q \chi_0 (n)$ by Fourier analysis on $\mathbb{Z}/q\mathbb{Z}$, 
  \begin{equation*}
       D_q(s)= \frac{\Lambda_{c_q+r} (s)}{\Lambda_{1} (s)}.
  \end{equation*}
  Hence we have the following relationship,
  \begin{equation}{\label{Functional equation for c_q}}
      \Lambda_{c_q+r}(s) = \epsilon \epsilon _1  (Nq^2)^{\frac{1}{2}-s} \overline{ \Lambda_{c_q+r}(1- \overline{s})   }.
  \end{equation}

For $q \equiv 1 \pmod{N}$ and $\chi$ a character mod $q$, define 
\[  f_{\chi} =  \sum_{ \substack{ a \Mod{q}  \\ (a,q)=1       } } \overline{\chi(a)} 
 f \bigg |  \begin{pmatrix}
 1 & \frac{a}{q} \\
0 & 1    
\end{pmatrix} + \mathbbm{1} _{ \chi=\chi_0} rf    \; \; \text{and}  \; \; \overline{f}_{\overline{\chi}} = \sum_{ \substack{ a \Mod{q}  \\ (a,q)=1       } } \chi(a) 
\overline{ f} \bigg |  \begin{pmatrix}
 1 & \frac{a}{q} \\
0 & 1    
                \end{pmatrix}  + \mathbbm{1}_{\chi=\chi_0} \overline{r} \overline{f}.       \]
Also define
 \[  C_{\chi} = \begin{cases}
 \overline{\epsilon}    \; \; \; \; \text{if} \;   \chi   \; \text{is trivial},\\ 
\chi(-N) \epsilon_1 \overline{  \epsilon_{\chi} \tau(\overline{ \chi})/\tau(\chi)    } \; \; \; \text{otherwise} .  \\

\end{cases}              \]
By substituting the Fourier expansion of $f$, we see $f_{\chi}$ has $n$th Fourier coefficient 
 \[        \begin{cases}    
 f_n( c_q(n)+r )  \; \; \; \; \text{if} \;   \chi   \; \text{is trivial},\\ 
  f_n \tau(\overline{\chi}) \chi(n) \; \; \; \text{otherwise},
 \end{cases}\]
 with a similar expression for $f_{\overline{\chi}}$.
  By  \eqref{Functional equation for c_q} and  the functional equation \eqref{Functional equation for twists} ,  Hecke's argument  implies we have the  modularity relationship
\begin{equation} {\label{Modular relationship for twists}}
 f _{\chi} \bigg| \begin{pmatrix}
 0 & -1 \\
Nq^2 & 0   
\end{pmatrix} = i^k \chi(-N) \epsilon_{\bold{1}}  \overline{C}_{\chi} \overline{f}_{ \overline{\chi}}. \end{equation} 
                If $\gamma, \gamma' \in \Gamma_0 (N)$ have the same top row, then it is easy to check that $ \gamma' \gamma ^{-1}$ is a power of $ \begin{pmatrix}
1 & 0 \\
N & 1 \\
\end{pmatrix}   $ .
 As $ \begin{pmatrix}
1 & 0 \\
N & 1 \\
\end{pmatrix} = H_N P^{-1} H_N^{-1} $  ,  $ f| \gamma$ depends only on the top row of $\gamma$. Let $\gamma_{q,a}$ denote any element of $\Gamma_0 (N)$ with top row $(q, -a)$. \\
Let $\gamma= \begin{pmatrix}
 q  & -b \\
 -Nm & r
\end{pmatrix}$ be an arbitrary element of $\Gamma_1(N)$. If $m=0$, then $\gamma$ is (up to sign if $N \leq 2$) a power of $P$, so $f|\gamma=f$. Otherwise, multiplying $\gamma$ on the left by $P^{-j}$ leaves $f|\gamma$ unchanged and replaces $q$ by $q+jmN$. By Dirichlet's theorem, we may assume that $q \equiv 1\pmod{N} $ is a prime and is large enough so Lemma \ref{Euler Factors} holds.   \\ \\
                
                    Equation \eqref{Modular relationship for twists}  implies the following.  
                \begin{align}{\label{Modularity on average}}
                & \sum_{ \substack{ a \Mod{q}  \\ (a,q)=1       } } C_{\chi} \overline{\chi(a)} 
 f \bigg |  \begin{pmatrix}
 1 & \frac{a}{q} \\
0 & 1    
                \end{pmatrix} + \mathbbm{1} _{\chi= \chi_0} \overline{ \epsilon} rf = C_{\chi} f_{\chi} = i^k \epsilon_{\bold{1}} \overline{f}_{\overline{\chi}} \bigg| \begin{pmatrix}
 0 & -1 \\
Nq^2 & 0   
\end{pmatrix}^{-1}  \nonumber \\
&= i^k \epsilon_{\bold{1}} \left(        \sum_{ \substack{ m \Mod q  \\ (m,q)=1       } }  \chi(-Nm) \overline{f}   \bigg|  \begin{pmatrix}
 1 & \frac{m}{q} \\
0 & 1   
\end{pmatrix}         \begin{pmatrix}
 0 & -1 \\
Nq^2 & 0   
\end{pmatrix}^{-1}     +   \mathbbm{1}_{ \chi= \chi_0} \overline{r} \overline{f}   \bigg|  \begin{pmatrix}
 0 & -1 \\
Nq^2 & 0   
\end{pmatrix}^{-1}    \right)   \nonumber   \\
&= \sum_{ \substack{ m \Mod q  \\ (m,q)=1       } } \chi(-Nm) f\bigg| 
\begin{pmatrix}
 0 & -1 \\
 N & 0 
\end{pmatrix}
\begin{pmatrix}
 1 & \frac{m}{q} \\
0 & 1   
\end{pmatrix}         \begin{pmatrix}
 0 & -1 \\
Nq^2 & 0   
\end{pmatrix}^{-1}     + \ind \overline{r} f \bigg |  \begin{pmatrix}
 q^2 & 0 \\
0 & 1    
                \end{pmatrix} \nonumber \\ 
&= \sum_{ \substack{ a \Mod q  \\ (a,q)=1       } } \overline{\chi(a)} f\bigg|  \gamma_{q,a} \begin{pmatrix}
 1 & \frac{a}{q} \\
0 & 1    
                \end{pmatrix} + \ind \overline{r} f \bigg |  \begin{pmatrix}
 q^2 & 0 \\
0 & 1    
                \end{pmatrix}  .
                \end{align}
                 Fix a residue $b$ coprime to $q$. 
                By orthogonality of Dirichlet characters and \eqref{Modularity on average}, 
                \begin{align*}  f\bigg|  \gamma_{q,b} \begin{pmatrix}
 1 & \frac{b}{q} \\
0 & 1    
                \end{pmatrix}&= \frac{1}{ \varphi(q)} \sum_{ \chi \Mod{q}} \chi(b) 
                \sum_{ \substack{ a \Mod q  \\ (a,q)=1       } } \overline{\chi(a)} f\bigg|   \gamma_{q,a} \begin{pmatrix}
 1 & \frac{a}{q} \\
0 & 1    
                \end{pmatrix}   \\ 
&= \frac{1}{\varphi(q)} \left( \sum_{ \chi \Mod q} \chi(b) C_{\chi}
                \sum_{ \substack{ a \Mod q  \\ (a,q)=1       } } \overline{\chi(a)} f\bigg|   \begin{pmatrix}
 1 & \frac{a}{q} \\
0 & 1    
                \end{pmatrix}   +    \overline{\epsilon} rf - \overline{r} f \bigg | 
             \begin{pmatrix}
 q^2 & 0 \\
0 & 1    
                \end{pmatrix}      \right)  .
                \end{align*}
                Replacing $a$ by $ab$ on the right hand side and writing 
                \[  \widehat{C_q}(a)=  \begin{cases}
                 \frac{1}{\varphi(q)} \sum_{ \chi \Mod{q}} C_{\chi} \overline{\chi(a)}  \; \text{if} \; \; (a,q)=1 , \\
                 \frac{\overline{\epsilon}r}{\varphi(q)} \; \; \text{otherwise},
                \end{cases}
                     \]
                     we obtain
                \begin{equation}{\label{formula for f|gamma}} f \bigg | \gamma_{q,b} =     \sum_{ a=0}^{q-1}  \widehat{C_q} (a)   f \bigg |   \begin{pmatrix}
 1 & \frac{(a-1)b}{q} \\
0 & 1    
                \end{pmatrix}    - \frac{\overline{r}}{\varphi(q)} f \bigg|  \begin{pmatrix}
 q^2 & -bq \\
0 & 1    
                \end{pmatrix}   .  \end{equation}
             Let 
                \[  S_q(x) =   \sum_{a=0}^{q-1}  \widehat{C_q} (a) e \left( \frac{(a-1)x}{q}    \right)           .    \]  From \eqref{formula for f|gamma}, the $n$th  Fourier coefficient of $f \bigg | \gamma_{q,b}$ is   $f_n S_q(bn)  - \frac{\overline{r}}{ \varphi(q)}q^k  \mathbbm{1} _{q^2 \nmid n}  f_{n/q^2}$. 
                \\
                 
                                We shall use the following Lemma. \begin{lem}{\label{lemma 2.3}}
                                Let $q  \equiv 1 \pmod{N}$ be a sufficiently large prime so Lemma \ref{Euler Factors} holds . For any $a$ such that $(a,q)=1$, there exists an $n \equiv a \pmod{q}$ such that $f_n \neq 0$. 
                                \end{lem}
                                
                                 \begin{proof}{ Proof of lemma \ref{lemma 2.3}.} \\
                 
             Let $ (a,q)=1$ and suppose there does not exist any $n \equiv a \pmod q$ such that $a_n \neq 0$.  Let $\chi_0$ be the trivial character mod $q$. By Fourier analysis, $\chi_0(n)= \frac{q-1 - c_q(n)}{q}$, so
                \[ \mathbbm{1}_ {n \equiv a}= \frac{ 1}{q}- \frac{c_q(n)+r}{q(q-1)} + \frac{r}{q(q-1)} + \frac{1}{q-1} \sum_{ \substack{ \chi \Mod q \\ \chi \neq \chi_0}} \overline{\chi(a)} \chi(n)   .  \]
                Multiplying by $a_n n^{-s}$ and summing implies
                \[  \left(  q-1+r    \right)    \Lambda_1(s)
                =    \Lambda_{c_q+r} (s) - q   \sum_{ \substack{ \chi \Mod q \\ \chi \neq \chi_0}} \overline{\chi(a)} \Lambda_{\chi} (s)    .    \]
                Using the functional equations \eqref{Functional equation for twists} and \eqref{Functional equation for c_q}, 
                \[ (q-1+r) \epsilon_{\bold{1}} N^{1/2-s} \overline{ \Lambda_1 (1- \overline{s})}      = \epsilon \epsilon_{\bold{1}} (Nq^2)^{1/2-s} \overline{ \Lambda_{c_q+r} (1- \overline{s})   }-  q  (Nq^2)^{1/2 -s}   \sum_{ \substack{ \chi \Mod q \\ \chi \neq \chi_0}} \epsilon_{\chi} \overline{\chi(a)} \; \; \overline{\Lambda_{\chi} (1- \overline{s})   }  .   \]
                Multiplying by $\overline{\epsilon_{\bold{1}}} (Nq^2)^{1/2-s}$, replace $s$ by $1- \overline{s}$ and conjugating, we get
                \[   (q-1+ \overline{r}) q^{1-2s} \Lambda_1 (s) = \overline{\epsilon} \Lambda_{c_q+r} (s) - q    \sum_{ \substack{ \chi \Mod q \\ \chi \neq \chi_0}} \chi(a) \epsilon_{\bold{1}} \overline{ \epsilon_{\chi}} \Lambda_{\chi} (s).                  \]
                 Comparing Dirichlet coefficients at $q$, either $a_q=0$ or $q-1+r=0$. The second equation 
                 implies $|\mu|=q$, which cannot happen as $F_q(s)$ has no poles for $\Re(s)\geq 1/2$. 
                 Comparing at $q^2$ now, $q(q-1+\overline{r}) = \overline{\epsilon} ( q-1+r) a_{q^2}$. Since $a_q=0$, $a_{q^2}=- \mu$, so this equation leads again to $|\mu|=q$, again giving a contradiction. 
               \end{proof}
                                 Since $q \equiv 1 \pmod{N}$, \eqref{modular with respect to gamma_{q,1}}  implies
                 $f|\gamma_{q,1}=f$. 
                           Using Lemma \ref{lemma 2.3} and equating Fourier coefficients of $f|\gamma_{q,1}$ and $f$  implies  $S_q(x)=1$ for all $x$ coprime to $q$.
                                For $n$ such that $q \nmid  n $, the $n$th Fourier coefficient of $f|\gamma_{q,b}$ is $f_n$. If $q|n$, the Fourier coefficient of $f|\gamma_{q,b}$ is independent of $b$, so $f|\gamma_{q,b}$ has the same $n$th Fourier coefficient as $f|\gamma_{q,1}$. Hence $f|\gamma_{q,b}=f$ for all $b$ coprime to $q$. It is also easily shown in  the case $b \equiv 0  \pmod{q}$, so  $f \in M_k ( \Gamma_1(N)) $. 
                                
                                \begin{rem}
                               For our proof we only need the fact $f \in M_k ( \Gamma_1(N)) $.  The above calculations imply that in fact $r=0$. \\               
                        
                                Notice that by definition $ S_q(0)=  \overline{ \epsilon } \left( \frac{  r}{ \varphi(q)} +1  \right)          $.
                        Using Fourier analysis on $\mathbb{Z}/ q \mathbb{Z}$
                            
                            \[ \widehat{C_q}(a+1)=  \frac{1 }{q} \sum_{x=0}^{q-1} S_q(-x) e( ax/q) =   \mathbbm{1} _{a=0} +   \frac{ S_q(0)-1}{q} .\] 
                            Hence $\widehat{C}_q(0)= \frac{ \overline{ \epsilon } r}{ \varphi(q)} = \frac{S_q(0)-1}{q}$.
                              Suppose $ \lambda \neq 0$, so $f_q \neq 0$. Then comparing   the $q$th Fourier coefficient of $f|\gamma_{q.1}$ and $f$ implies  $ S_q(0)=1 $, i.e $r=0$. The fact $S_q(0)=1, r=0$ also shows $\epsilon=1$ and  $ |\mu|=1$,    so $\lambda$ is real in this case.
                              \\
                              Now suppose $\lambda=0$. Then $f_{q^2} = -\mu q^{k-1}$, $f_1=1$. Comparing coefficients at $q^2$ now, 
                              \[        \mu q^{k-1}= \mu q^{k-1} \left( \frac{ \overline{ \epsilon} r q }{ \varphi(q)}+1      \right)  + \frac{\overline{r} q^k}{ \varphi(q)} ,                       \]
                               which implies $|\mu|=1$. Hence $ \epsilon= \mu$, so $r=0$ again. 
                        \end{rem}
                
                 \section{Proof of theorem \ref{Main thm}.}
                 \begin{proof}
                 We use the following argument from \cite{Converse}. 
                Let $C$ denote the set of normalized Hecke eigenforms of weight $k$ and conductor dividing $N$, and for $g \in C$ with Fourier expansion $\sum_{n=1}^{\infty} g_n e(nz)$, let $L_g(s) = \sum_{n=1}^{\infty} g_n n^{-s - \frac{k-1}{2}}$. \\
                
                  Let $X$ denote the set of pairs $(\xi_1, \xi_2)$ where $\xi_1 \pmod{N_1}$, $\xi_2 \pmod{N_2}$ are primitive Dirichlet characters such that $N_1N_2 \mid N$, $\xi_1(-1) \xi_2 (-1)=(-1)^k$ and if $k=1$ then $\xi_1(-1)=1$. Let 
                  \begin{equation*} L_{\xi_1, \xi_2}(s)= L(s+ \frac{k-1}{2} , \xi_1) L(s- \frac{k-1}{2}, \xi_2)          \end{equation*}
                  where the factors on the right are the usual Dirichlet L functions. \\ Since $f \in M_k (\Gamma_1(N))$, by newform theory and the description of Eisenstein series in \cite[Chapter~4.7]{miyake2006modular}, there are Dirichlet polynomials $D_{\xi_1, \xi_2}$ and $D_g$ such that 
                  \begin{equation}{\label{Newform}} L_f(s)= \sum_{(\xi_1, \xi_2) \in X}  D_{\xi_1, \xi_2} L_{\xi_1, \xi_2}(s) + \sum_{ g \in C}  D_g(s) L_g(s)       .     \end{equation}
                 
                Furthermore the coefficients of each Dirichlet polynomial are supported on divisors of $N$.   Let us define $F_{p,g}(s)$ to be the Euler factor of $L_g(s)$ at $p$, where $g \in  C$. Also, let $F_{p, (\xi_1, \xi_2)} (s)$ be the Euler factor of $L_{\xi_1, \xi_2} (s) $ at p.\\
                   We will say that the Euler products of two L functions $L_1, L_2$  are equivalent if their Euler factors are the same except for finitely many primes and inequivalent otherwise.  The Ranking--Selberg method  ( see \cite[Corollay~4.4]{booker2016converse}) implies that the $L$ functions on the right hand side of \eqref{Newform} are pairwise inequivalent. Combining this with the linear independence result in \cite[Theorem~2]{kaczorowski1999linear}, we see the right hand side of \eqref{Newform} has exactly one non-zero term. Hence $L_f(s)= D_{y}(s) L_{y}(s)$ for some $y \in C \cup X$. \\ \\
                 In either case 
                 \[ D_{y} (s) = \prod_{p|N} \frac{F_p(s)}{ F_{p,y} (s)}                  \]
                and $D_y$ satisfies a functional equation of the form 
                 $  D_y(s) = \varepsilon_y  (N_{y})^{\frac{1}{2}-s} \overline{D_y ( 1- \overline{s})} $ where $N_y$ is a positive integer, $ |\varepsilon_y|=1.$ The $\mathbb{Q}$ linear independence of $\log p$ for primes $p$ and the fact $D_{y}(s)$ is entire implies that $ \frac{F_p(s)}{F_{p,y}(s)}$ is entire for $p \mid N$ and its zeroes are symmetric with respect to the line $\Re(s)= \frac{1}{2}$.  \\ \\
                  Suppose $ y=(\xi_1, \xi_2) \in X$. \\ 
                  If $k 
                 \geq 2$, for $p>N$, $F_p(s) =  F_{p,y}(s)    $  has a pole on the line 
                 $ \Re(s) =  \frac{k-1}{2} \geq \frac{1}{2}$.  This is a contradiction to our assumptions on $F_p(s)$ in the statement of the Theorem \ref{Main thm}. Hence $k=1$. \\
                 When $k=1$, if  $p \mid N$,  $\frac{F_p(s)}{F_{p,y}(s)}$ has no zeroes for $\Re(s) \geq \frac{1}{2}$. The symmetry of the zeroes around $\Re(s)=1/2$ then implies there are in fact no zeroes.  In particular, $D_y(s)$ has no zeroes and the fundamental theorem of algebra implies it must be constant. Equating coefficients of $L_f(s)$ and $L_y(s)$  implies $D_y(s)=1$. The functional equation then implies $N=N_1N_2$ as required.  \\  \\
                  Suppose $y=g \in C$. \\
                 Deligne showed that $L_y$ satisfies the Ramanujan hypothesis, so $ L_y(s)$ lies in the Selberg class. Convergence of the logarithm in the definition of the Selberg class means  $F_{p,y} (s) $ has no zeroes and is analytic for $\Re(s) \geq 1/2$. Repeating the argument above, $D_y(s)=1$ and the conductor of $g$ equals $N$ by the functional equation as required. 
                 
                 \end{proof}
                 We now deal with the case that $\Lambda_1(s)$ is not entire. We shall deal with this issue as in \cite{Converse} by twisting away the poles. \\  Fix a prime $q \nmid N$ and a primitive character $\chi \pmod{q}$ and consider the twisted sequence $a_n'= a_n \chi(n)$ in place of $a_n$ and $Nq^2$ in place of $N$. Then all the hypotheses of Theorem \ref{Main thm} are satisfied, however now our associated L function $\Lambda_1(s)$ is entire. Hence, either there is a primitve cuspform $f'$ of conductor $Nq^2$ with Fourier coefficients $a_n' n^{  \frac{k-1}{2}  }$, or $k=1$ and there are primitive characters $\xi_1'$, $\xi_2'$ such that $a_n'= \sum_{d|n} \xi_1' (n/d) \xi_2'(d)$. \\
                 We deal with the cuspidal case first. By newform theory \cite[Theorem~3.2]{atkin1978twists}, we can twist $f'$ by $\overline{\chi}$ to get a primitive cuspform $f$ of conductor $Nq^j$, for some $j$ with Fourier coefficients $ a_n' \overline{\chi(n)} n^{   \frac{k-1}{2}  }= a_n n^{   \frac{k-1}{2} }$ for every $n$ coprime to $q$. Applying this argument to two different choices of $q$, strong multiplicity one implies $f$ has conductor $N$ and Fourier coefficients $a_n n^{\frac{k-1}{2}}$ as wanted. \\
                 In the non cuspidal case, let $\xi_i  \pmod{N_i}$ for $i=1,2$ be the primitive character inducing $ \xi_i' \overline{\chi}$. Then 
                 \begin{equation}{\label{Dirichlets theorem}}
                     a_p = \xi_1(p)+ \xi_2(p) \; \; \; \text{for all sufficiently large primes} \; \; p.
                 \end{equation} The characters $\xi_1$ and $\xi_2$ have opposite parity so we  normalize $\xi_1$ to be even. Dirichlet's theorem on primes in arithmetic progressions implies $\xi_1, \xi_2$ are uniquely defined by \eqref{Dirichlets theorem}. Again, using two choices of $q$, we have $N_1N_2= N$ and $ a_n = \sum_{d|n} \xi_1(n/d) \xi_2 (d)$. Hence we have completed the proof. 
                  
                 \section{Acknowledgements}
                 I would like to thank my supervisor Andrew Booker for all his help on this work. His suggestions and ideas have been incredibly useful.

                \bibliographystyle{amsplain}
\bibliography{sample}

\providecommand{\bysame}{\leavevmode\hbox to3em{\hrulefill}\thinspace}
\providecommand{\MR}{\relax\ifhmode\unskip\space\fi MR }
\providecommand{\MRhref}[2]{%
  \href{http://www.ams.org/mathscinet-getitem?mr=#1}{#2}
}
\providecommand{\href}[2]{#2}
\begin{thebibliography}{10}

\bibitem{atkin1978twists}
A.~O.~L. Atkin and Wen Ch'ing~Winnie Li, \emph{Twists of newforms and
  pseudo-eigenvalues of {$W$}-operators}, Invent. Math. \textbf{48} (1978),
  no.~3, 221--243. \MR{508986}

\bibitem{Converse}
Andrew~R. Booker, \emph{A converse theorem without root numbers}, Mathematika
  \textbf{65} (2019), no.~4, 862--873. \MR{3952509}

\bibitem{booker2014weil}
Andrew~R. Booker and M.~Krishnamurthy, \emph{Weil's converse theorem with
  poles}, Int. Math. Res. Not. IMRN (2014), no.~19, 5328--5339. \MR{3267373}

\bibitem{booker2016converse}
\bysame, \emph{A converse theorem for {${\rm GL}(n)$}}, Adv. Math. \textbf{296}
  (2016), 154--180. \MR{3490766}

\bibitem{kaczorowski1999linear}
J.~Kaczorowski, G.~Molteni, and A.~Perelli, \emph{Linear independence in the
  {S}elberg class}, C. R. Math. Acad. Sci. Soc. R. Can. \textbf{21} (1999),
  no.~1, 28--32. \MR{1669479}

\bibitem{kaczorowski2018hecke}
J.~Kaczorowski and A.~Perelli, \emph{On a {H}ecke-type functional equation with
  conductor {$q=5$}}, Ann. Mat. Pura Appl. (4) \textbf{197} (2018), no.~6,
  1707--1728. \MR{3855407}

\bibitem{kaczorowski2020classification}
J~Kaczorowski and A~Perelli, \emph{Classification of l-functions of degree 2
  and conductor 1}, arXiv preprint arXiv:2009.12329 (2020).

\bibitem{miyake2006modular}
Toshitsune Miyake, \emph{Modular forms}, english ed., Springer Monographs in
  Mathematics, Springer-Verlag, Berlin, 2006, Translated from the 1976 Japanese
  original by Yoshitaka Maeda. \MR{2194815}

\bibitem{neururer2018weil}
Michael Neururer and Thomas Oliver, \emph{Weil's converse theorem for {M}aass
  forms and cancellation of zeros}, Acta Arith. \textbf{196} (2020), no.~4,
  387--422. \MR{4164486}

\bibitem{perelli2005survey}
Alberto Perelli, \emph{A survey of the {S}elberg class of {$L$}-functions.
  {I}}, Milan J. Math. \textbf{73} (2005), 19--52. \MR{2175035}

\end{thebibliography}
 
  \textbf{School of Mathematics, University of Bristol, Bristol, BS8 1UG, United Kingdom }
  \\
 \textbf{Email address: michaelfarmer868@gmail.com}

\end{document}